\documentclass[11pt]{article}

\usepackage{amsmath}
\usepackage{amssymb}
\usepackage{epsf}
\usepackage{psrotate}
\usepackage{dgstpp}
\usepackage{pproof}
\usepackage{remexpp}
\usepackage{mathrsfs}
\let\rscr=\mathscr
\let\mathscr=\relax
\let\mcal=\mathcal
\usepackage{eucal}
\let\escr=\mathcal
\let\mathcal=\relax
\usepackage{accents}

\newtheorem{theorem}{Theorem}[section]
\newtheorem{proposition}[theorem]{Proposition}
\newtheorem{lemma}[theorem]{Lemma}
\newtheorem{corollary}[theorem]{Corollary}
\newremark{definition}[theorem]{Definition}
\newremark{example}[theorem]{Example}
\newremark{remark}[theorem]{Remark}

\newcommand{\av}{a_V}
\newcommand{\bdel}[1]{{\pbar\nabla}_{\!#1}}
\newcommand{\bU}{\bar{U}}
\newcommand{\bV}{\bar{V}}
\newcommand{\cD}{\mcal{D}}
\newcommand{\cJ}{\escr{J}}

\newcommand{\vl}{{\scriptscriptstyle\mathrm{V}}}
\newcommand{\cT}{\mcal{T}}
\newcommand{\dc}{\dot{c}}
\newcommand{\ddc}{\ddot{c}}
\newcommand{\dx}{\dot{x}}

\newcommand{\ddx}{\ddot{x}}

\newcommand{\del}[1]{\nabla_{\!#1}}
\newcommand{\ds}{\oplus} 
\newcommand{\eG}{\escr{G}}
\newcommand{\D}{\nabla}

\newcommand{\fj}{\text{fix}\mspace{1mu}J}
\newcommand{\g}{\gamma}
\newcommand{\G}{\Gamma}
\newcommand{\gl}{\mathfrak{gl}}
\newcommand{\half}{\mbox{$\textstyle\frac{1}{2}$}}
\newcommand{\iso}{\cong}
\newcommand{\J}{\escr{J}}
\newcommand{\K}{\escr{K}}

\newcommand{\lsp}{[\kern-0.15em[} 
\newcommand{\nd}{\vdash}
\newcommand{\p}{\pi}
\newcommand{\pbar}[1]{\begingroup\let\hidewidth\relax\accentset{\hrulefill}{#1}\endgroup}
\newcommand{\rsp}{]\kern-0.15em]} 
\newcommand{\surj}{\rightarrow\kern-.82em\rightarrow}

\newcommand{\V}{{\rscr V}}
\newcommand{\ve}{\varepsilon}

\newcommand{\R}{\mathbb{R}}

\renewcommand{\H}{\rscr{H}}
\renewcommand{\l}{\lambda}
\renewcommand{\P}{\Phi}
\renewcommand{\S}{\lower .2ex\hbox{$\escr S$}}

\makeatletter

\@addtoreset{equation}{section}

\newcommand{\Ad}[1]{\mathop{\operator@font Ad}\nolimits_{#1}}
\newcommand{\Aut}{\mathop{\operator@font Aut}\nolimits}
\newcommand{\con}{\mathop{\operator@font con}\nolimits}
\newcommand{\diag}{\mathop{\operator@font diag}\nolimits}
\newcommand{\dom}{\mathop{\operator@font dom}\nolimits}
\newcommand{\econ}{\mathop{\operator@font EConn}\nolimits}
\newcommand{\End}{\mathop{\operator@font End}\nolimits}
\newcommand{\Hom}{\mathop{\operator@font Hom}\nolimits}
\newcommand{\im}{\mathop{\operator@font im}\nolimits}
\newcommand{\op}{\mathop{\operator@font Op_2}\nolimits}
\newcommand{\opa}{\mathop{\operator@font Op_2^{alt}}\nolimits}
\newcommand{\Ric}{\mathop{\operator@font Ric}\nolimits}
\newcommand{\qsp}{\mathop{\operator@font QSpray}\nolimits}
\newcommand{\sode}{\mathop{\operator@font DE_2}\nolimits}
\newcommand{\tr}{\mathop{\operator@font tr}\nolimits}
\makeatother

\font\heads = cmbx12

\preprint{DRP5}
\title{General Connections, Exponential Maps,\\
   and Second-order Differential Equations}
\author{L. Del Riego}
\address{Facultad de Ciencias\\
   Zona Universitaria\\
   Universidad Aut\'onoma de San Luis Potos\'\i\\
   San Luis Potos\'\i, SLP\\
   78290 MEXICO\\
   lilia@fc.uaslp.mx}
\author{Phillip E. Parker}
\address{Mathematics Department\\
   Wichita State University\\
   Wichita KS 67260-0033\\
   USA\\
   phil@math.wichita.edu}
\date{6 July 2011}
\abstract{
The main purpose of this article is to introduce a comprehensive, unified
theory of the geometry of all connections.  We show that one can study a
connection \emph{via} a certain, closely associated second-order
differential equation, its geodesic quasispray.  One of the most important
results is our extended Ambrose-Palais-Singer correspondence.  We extend
the theory of geodesic sprays to the quasisprays, show that locally
diffeomorphic exponential maps can be defined for any SODE, and give a
full theory of (possibly nonlinear) covariant derivatives for (possibly
nonlinear) connections.  In the process, we introduce \emph{vertically
homogeneous} connections.  Unlike homogeneous connections, these complete
our theory and allow us to include Finsler spaces in a completely
consistent manner.

This is an expanded version of the article published in {\it Differ.\
Geom.\ Dyn.\ Syst.}\ {\bf 13} (2011) 72--90.  Included are the proof
published in {\it Nonlinear Anal.}\ {\bf 63} (2005) e501--e510 and some
new material on homogeneity.}

\msc{53C05}{53C15, 53C22}

\begin{document}

\maketitle

\setcounter{page}{0}\thispagestyle{empty}\strut\vfill\eject

\section{\heads Introduction}
In modern geometry, there are various kinds of connections for a given
manifold $M$ with a bundle structure over it.  For example:
\begin{itemize}

\item A \textit{Cartan} connection may be considered as a version of the
general concept of a principal connection, in which the geometry of the
principal bundle is tied to the geometry of the base manifold \cite{C,S}.
Cartan connections describe the geometry of manifolds modelled on
homogeneous spaces. Under certain technical conditions, they can be related
to the remaining types \cite{S}.

\item A \textit{general} connection on any fibre bundle $E\surj
M$ is a splitting of $TE$ into the natural vertical bundle and a
\textit{horizontal} bundle \cite{E}.  If the splitting is equivariant for
the structure group (or, more generally, some subgroup) $G$, then it defines
an Ehresmann \textit{$G$-connection} \cite{E,P}.

\item A \textit{principal} connection is an Ehresmann $G$-connection on a
principal $G$-bundle $(P,M,G)$ \cite{E,P}.

\item A \textit{linear} connection on a vector bundle $(E,M,V,GL(V))$ over
$M$ with model fiber $V$ is associated to a principal connection on the
frame bundle with group $GL(V)$ \cite{E,P}.  All others are
\textit{nonlinear}, among which are the \textit{affine} connections with $G
= A_n$.  It is unfortunate that in the extant literature on nonlinear
connections, for example \cite{Kw,Bar,V1,V2,D} all written well after
\cite{E}, a nonlinear connection is defined to be a particular highly
restricted type of connection on $TM -0$.

\item A \textit{Koszul} connection is a linear operator of the type of a
covariant derivative on a vector bundle. It gives rise to a linear
connection on the vector bundle \cite{P}.

\end{itemize}
We are only concerned with finite-dimensional real vector bundles $E$
(vector spaces $V$), so $GL(V) \cong GL(n,\R) = GL_n$ with $n = \dim V$.
Moreover, our only direct concern is when $E = TM$, so the principal bundle
is $LM$, the bundle of linear frames, $n = \dim M$, and the connections are
$G$-connections for a suitable subgroup $G\le GL_n$.  All
\textit{pseudoRiemannian} connections are linear connections of this last
type \cite{O,P}.

Since the fundamental work of Ehresmann \cite{E}, we have had a consistent
terminology for connections on a manifold $M$.  A connection on $M$ is a
splitting $TTM =\V\ds\H$ where $\V$ is the natural vertical bundle and $\H$
is a complementary subbundle, the horizontal bundle.  In this article, we
continue our study of smooth general connections on the tangent bundle
$TM$ of a smooth, paracompact, connected manifold $M$.  We shall use
``nonlinear'' in the original sense of Ehresmann.

Let us note that Bucataru and Miron \cite{BM} recently defined a completely
different kind ofnonlinear connection \textit{via} a generalization of the
Koszul procedure. They begin with the {\em assumption\/} that
parallel transport is to be linear, construct from that a nonlinear
covariant derivative operator, and thence a nonlinear connection. We do not
begin with that, or any other such, assumption; instead, we begin with an
arbitrary (smooth) nonlinear connection, and then construct a
nonlinear covariant derivative operator {\em via\/} an extension of the
connector procedure (Def.\,\ref{kap}).

The geodesic spray in pseudoRiemannian geometry, the integral curves of
which are the geodesics of the Levi-Civita connection, has played an
important role; see, for example, \cite{BC,B}.  Riemannian geometry has
been a main thread of mathematics over the last century \cite{Ber}, and
Finsler geometry has recently undergone somewhat of a revival \cite{A}.

Second-order differential equations (SODEs) are an important class of
vector fields on the tangent bundle.  Our principal motivation for this
work was the desire to make a comprehensive theory of the geometry of
nonlinear connections and SODEs which would include (pseudo)Riemannian
geodesic sprays and analogues for Finsler-like spaces as examples.
Moreover, such a theory would also apply to the geometry of principal
symbols of PDOs \cite{P8} and to stability problems around linear
connections; {\em e.g.,} \cite{BP4,BP6}.

Section \ref{rev} contains our notation, conventions, and a summary of our
earlier article \cite{DRP1}.  In Section \ref{exp} we present the new
exponential maps defined by SODEs.  Section \ref{cs} describes the
relations among (possibly nonlinear) connections, certain SODEs
(quasisprays), the associated (possibly nonlinear) covariant
derivatives, and geodesics.  It also contains the various parts of our
extended Ambrose-Palais-Singer (APS) correspondence.  In Section \ref{fs}
we provide a simple example using Finsler spaces.  Finally, Section
\ref{gcs} begins with the extension of the main results of \cite{BP6} to
SODEs, using our new, extended construction of exponential maps.  It also
includes the extension of the main stability result of \cite{BP4,DRP1} to
all SODEs.

The authors thank CONACYT and FAI for travel and support grants, Wichita
State University and Universidad Aut\'onoma de San Luis Potos\'{\i} for
hospitality during the progress of this work, and J. Hebda and A. Helfer
for helpful conversations.  Del Riego also thanks M. Mezzino for writing a
Mathematica package for her use.

\section{\heads Review and definitions}\label{rev}
A second-order differential equation (SODE) on a manifold $M$ is defined
as a projectable section of the second-order tangent bundle $TTM \surj TM$
\cite{BC,B,BJ}.  Recall that an integral curve of a vector field on $TM$
is the canonical lift of its projection if and only if the vector field is
projectable \cite{BC}.  For a curve $c$ in $M$ with tangent vector field
$\dot c$, this $\dot c$ is the canonical lift of $c$ to $TM$ and $\ddot c$
is the canonical lift of $\dot c$ to $TTM$.  Then each projectable vector
field $S$ on $TM$ determines a second-order differential equation on $M$
by $\ddot c = S\circ \dot c$, and each such curve with $\dot c(s_{0}) =
v_{0}\in T_{c(s_{0})}M$ is a solution with initial condition $v_0$.
Solutions are preserved under translations of parameter, they exist for
all initial conditions by the Cauchy theorem, and, as our manifolds are
assumed to be Hausdorff, each solution will be unique provided we take it
to have maximal domain; {\em i.e.,} to be inextendible \cite{BC,DRD,HS}.

There are two vector bundle structures on $TTM$ over $TM$, denoted
here by $\pi_T$ and $\pi_*$.  Let $J$ be the canonical involution on
$TTM$, so it isomorphically exchanges the two vector bundle structures on
$TTM$.  We denote the fixed set of $J$ by $\fj$ and observe that it is
an {\em affine\/} subbundle of both $\pi_T$ and $\pi_*$, but not a {\em
vector\/} subbundle of either.
\begin{definition}\label{dfs}
A section $S$ of $TTM$ over $TM$ is a SODE when $JS = S$, or equivalently,
when $S\in\G(\fj)$. The space of all SODEs is denoted by $\sode(M)$,
and those vanishing on the 0-section of $TM$ by $\qsp(M)$.
\end{definition}
Thus a SODE can be expressed locally as $S:(x,y)\mapsto (x,y,y,\S(x,y))$.
\begin{remark}
If desired, one may work with jet spaces using $J^1(\R_{\,0},M)\cong TM$
and $J^2(\R_{\,0},M) \cong \fj$, where the notation indicates jets with
fixed source $0\in\R$ and target any point in $M$.
\end{remark}

The vertical bundle $\V = \ker(\pi_* :TTM\surj TM)$ is a vector subbundle
with respect to both vector bundle structures on $TTM$.  In induced local
coordinates, elements of $\V$ look like $(x,y,0,Y)$.  We observe that
$\fj$ is an affine subbundle of $TTM$ with translation vector bundle $\V$.
This allows us to regard $\sode(M)$ as an affine space with translation
vector space $\G(\V)$ and with $\qsp(M)$ as a closed affine subspace, so
that both are affine nuclear Fr\'echet spaces \cite{T}.

Before commenting further on this definition, we must briefly digress to
consider the notion of homogeneity for functions.

Consider the equation $f(ax) = a^m f(x)$.  In projective geometry, for
example, one usually requires this to hold only for $a\ne 0$.  We shall
call this {\em projectively homogeneous\/} of degree $m$.  In other areas,
such as Euler's Theorem in analysis, one further restricts to $a>0$.  We
shall call this {\em positively homogeneous\/} of degree $m$.  Finally, in
order that homogeneity of degree 1 coincide with linearity, one must allow
all scalars $a\in\R$ (including zero).  We shall call this {\em completely
homogeneous\/} of degree $m$. By $h(m)$ we shall mean complete homogeneity
on $TM$ and projective homogeneity on $TM-0$.

The difference between projective homogeneity and complete homogeneity is
minor; essentially, it is just the difference between working on $TM-0$
and on $TM$.  The difference between positive homogeneity and the other
two is more significant.  For example, the inward-going and outward-going
radial geodesics of the Finsler-Poincar\'e plane in \cite{BCS} have
different arclengths.

We must distinguish carefully between parametrized \emph{curves} and
unparametrized \emph{paths}.  A \emph{path} is the image of a parametrized
curve.  Alternatively, one may identify paths with equivalence classes
of curves:  two curves are equivalent if and only if they are
reparametrizations of each other.  This is clearly a bijective
correspondence, as each equivalence class determines a unique path (the
common image of all curves in the class) and conversely.

Recall that there are natural vector bundle maps $\K:\V\to TM$, respecting
$\pi_T$, and $\J:\left(\pi^*TM \iso TM\ds TM\right)\to \V$ which are
isomorphisms on fibers.  Both are versions of canonical parallel
translation on a vector space.  Let $S$ be a SODE over $M$, $p$ a point in
$M$, and consider the value $S(0)$ for 0 in $T_p M$, a vertical vector in
$T_0 T_p M$. Define a vertical vector field by
\begin{equation}
R(u) = \J_u\K(S(0))
\label{vcf}
\end{equation}
for each $u\in T_p M$ and for each $p \in M$.  Note that $R$ is
\emph{vertically constant} as it is constant along the fibers of $TM$ in
the obvious sense.  Clearly, $Q = S-R$ is a quasispray.  Moreover, $R$ is
the vertical lift $U^\vl$ of a vector field $U$ on $M$ as is immediate
from the definition \cite[p.\,6f\,]{YI}. We may think of $R$ or $U$ as an
external force, such as a wind.

We use $\g_u$ to denote the unique inextendible $S$-geodesic with initial
velocity $u\in T_p M$, as in \cite{O}.  Now we are ready to consider
homogeneity for SODEs. Noting that any reasonable notion of homogeneity
will force $S$ to be a quasispray and taking into account the
decomposition just established, we may as well consider only quasisprays.

\begin{definition}\label{nhq}
A quasispray $Q$ is \emph{homogeneous} if and only if for each $0\neq u\in
TM$ and all scalars $a\neq 0$, all of the curves $\g_{au}$ determine the
same unique path in $M$.
\end{definition}

Associated with each quasispray is its system of nondegenerate integral
curves.  A homogeneous qspray gives rise to a system of paths in the sense
of Douglas \cite{jD}, who showed that any such system can be obtained as
the paths of the integral curves of a SODE that is $h(2)$.  (Note that of
all possible $h(m)$, only $h(2)$ is invariantly well-defined globally on
$TM$.)  We extend our definition of homogenity to systems of paths in the
obvious way, so that a system of paths in the sense of Douglas becomes a
homogeneous system of paths in our sense.

We are interested primarily in general connections and their derived
quasisprays. Thus we are interested in systems of (parametrized) curves so
as to include those that arise from inhomogeneous qsprays. It follows that
any homogeneous system of paths (a system in the sense of Douglas) is a
system of paths in our sense but not conversely; we include inhomogeneous
systems while Douglas excluded them.

The following condition is sufficient, but not necessary, for homogeneity
as just defined.  Denote scalar multiplication in the vertical bundle $\V$
by $a_V$.
\begin{definition}\label{hms}
We say that a SODE $S$ is $h(m)$ when $S(av) =
a_*\av^{m-1}S(v)$.
\end{definition}
Explicitly, $a_* \av^{m-1} (x,y,X,Y) = (x,ay,aX,a^m Y)$ in induced local
coordinates.  In other words, the functions $\S(x,y)$ are completely
(respectively, projectively) homogeneous of degree $m$ in the vertical
component in \emph{some} induced local coordinates:  $\S(x,ay) =
a^{m}\S(x,y)$ for some $m\ge 1$ (respectively, $m<1$) and all scalars
$a\in\R$ (respectively, $a\ne 0$).  Note that $h(m)$ SODEs on $TM$
vanish on the 0-section, so are quasisprays.

The break comes at $m=1$ because an $h(m)$ SODE is to be associated with a
connection whose homogeneity formula effectively contains $a^{m-1}$; see
Proposition \ref{hgs}.  In \emph{some} induced local coordinates,
$S:(x,ay)\mapsto (x,ay,ay,a^m \S(x,y))$.

\begin{remark}
Let $C$ denote the Euler-Liouville vector field on $TTM$.  We recall that
in local coordinates, $J(x,y,X,Y) = (x,X,y,Y)$ and $C:(x,y)\mapsto
(x,y,0,y)$.  In the extant literature \cite{DRP1,G,G1,KV,LR}, one finds
homogeneous {\em vector fields\/} of degree $m$ defined by $[C,S] =
(m-1)S$.  In any local coordinates, $S:(x,ay)\mapsto (x,ay,a^{m-1}y,a^m
\S(x,y))$.  It follows that a homogeneous SODE in our theory can be a
homogeneous vector field only for $m=2$.
\end{remark}
Hereinafter we shall call $h(2)$ SODEs {\em quadratic\/} sprays, in
agreement with \cite{G1,KV,LR}.  (Note that {\em complete\/} homogeneity
is required for our quadratic sprays to coincide with the usual spray of
\cite{APS}.)  We denote the set of SODEs on $M$ that are $h(m)$ by
${\qsp}_m(M)$.  It has been usual to consider only (positive) integral
degrees of homogeneity, but we make no such restriction.

Elsewhere \cite{KV}, projectable vector fields on $TM-0$ are called {\em
semi\-sprays\/} and the name {\em sprays\/} (confusingly) used for those
that are $h(2)$ on $TM-0$.  We will associate a SODE to each (possibly
nonlinear) connection in the role of a {\em geodesic spray\/} (see
Theorems \ref{cis} and \ref{cg=sg}), so we shall use the name
``quasispray" to reflect this new, extended role (and to distinguish ours
from all the others; {\em e.g.,} \cite{R}).  We do, however, explicitly
consider only smooth SODEs defined on the entire tangent bundle $TM$;
others \cite{A,BCS,KV} use only the reduced tangent bundle with the
0-section removed, which is necessary when considering $h(m)$ SODEs when
$m<1$ (including $m<0$).  In general, one usually requires SODEs to be at
least $C^0$ across the zero-section when possible; {\em e.g.,} for Finsler
spaces.  Most of our results are easily seen to hold {\em mutatis
mutandis\/} in these cases as well; any unobvious exceptions will be noted
specifically.

As we said, the desire to include Finsler spaces consistently was one of
our motivations.  What {\em should\/} be the Finsler-geodesic ``spray''
associated with a Finsler metric tensor is {\em not\/} a homogeneous
vector field, but an $h(1)$ SODE in our theory; see \cite{DR1} for related
results.  However, the Finsler geodesic coefficients have both $h(2)$ and
$h(1)$ parts, making what we shall see in Section \ref{fs} is an $h(1)$
semispray.

Several important results concerning quadratic sprays \cite{APS,BC,D,KV}
rely on the facts that each such spray $S$ determines a unique torsion-free
linear connection $\Gamma$, and conversely, every quadratic spray $S$ arises
from a linear connection $\Gamma$ the torsion of which can be assigned
arbitrarily.  The solution curves of the differential equation $\ddot c =
S_{\Gamma}\circ \dot c$ for a connection-induced spray are precisely the
geodesics of that (linear) connection.  These solution curves are not only
preserved under translations, as is true in general, but also under affine
transformations of the parameter $s \mapsto as+b $ for constants $a,b$ with
$a\neq 0$.  Note that, with our definition, the latter also holds
for homogeneous SODEs.

In the general case, a (possibly nonlinear) connection $\Gamma$ gives rise
to a quasispray $S$ (see Proposition \ref{cis}), but the correspondence
has not been studied before.  We shall extend most of the preceding
features of the quadratic spray--linear connection correspondence to the
general setting.  One of our ultimate goals is to determine just how well
nonlinear connections can be studied {\em via\/} their quasisprays.

We continue with the principal definitions. Let $S$ be a SODE on $M$.
\begin{definition}
We say that a curve $c:(a,b)\rightarrow M$ is a {\em geodesic\/} of $S$ or
an {\em $S$-geodesic\/} if and only if the natural lifting $\dot c$ of $c$
to $TM$ is an integral curve of $S$.
\end{definition}
This means that if $\ddot c$ is the natural lifting of $\dot c$ to $TTM$,
then $\ddot c = S(\dot c)$ is the $S$-geodesic equation.
\begin{definition}
We say that $S$ is {\em pseudoconvex\/} if and only if for each compact $K
\subseteq M$ there exists a compact $K^\prime \subseteq M$ such that each
$S$-geodesic segment with both endpoints in $K$ lies entirely within
$K^{\prime}$.
\end{definition}
If we wish to work directly with the integral curves of $S$, we merely
replace ``in''  and ``within'' by ``over''.
\begin{definition}
We say that $S$ is {\em disprisoning\/} if and only if no inextendible
$S$-geodesic is contained in (or lies over) a compact set of $M$.
\end{definition}
In relativity theory, such inextendible geodesics are said to be imprisoned
in compact sets; hence our name for the negation of this property.

Following this definition, we make a convention:  all $S$-geodesics are
always to be regarded as extended to the maximal parameter intervals ({\em
i.e.,} to be inextendible) unless specifically noted otherwise.  When the
SODE $S$ is clear from context, we refer simply to geodesics.  Note that
no SODE can be disprisoning on a compact manifold.  However,
Corollary~\ref{qcov} may be used to obtain results about compact manifolds
for which the universal covering is noncompact.

We refer to \cite{DRP1} for motivation, further general results, and
results specific to homogeneous SODEs (called homogeneous sprays there),
and to \cite{DRP2} for more examples.  Note that the SODEs in \cite{DRP1}
were positively homogeneous; the extension of those results to complete
homogeneity is straightforward, once the definition of homogeneous spray
there is corrected to the one for homogeneous SODE here.

\section{\heads Exponential maps}\label{exp}
Let $S$ be a SODE on $M$. We define the generalized exponential map{\em
s\/} (plural!) $\exp^\ve$ of $S$ as follows.

First let $p\in M$, $v\in T_p M$, and $c$ be the
unique $S$-geodesic such that
\begin{eqnarray*}
\ddc &=& S(\dc)\\
c(0) &=& p\\
\dc(0) &=& v
\end{eqnarray*}
Define
$$\exp^\ve_p(v) = c(\ve)$$
for all $v\in T_p M$ for which this makes sense.  From the existence of
flows ({\em e.g.,} \cite[p.\,175]{HS}), it follows that this is well
defined for all $\ve$ in some open interval $(-\ve_p,\ve_p)$, which in
general depends on $p$, and for all $v$ in some open neighborhood $U_p$ of
$0\in T_p M$, which in general depends on the choice of
$\ve\in(-\ve_p,\ve_p)$.  This defines $\exp^\ve_p$ at each $p\in M$.
\begin{remark}
On $TM-0$, it is frequently convenient to define $\exp^\ve_p(0) = p$. One
must then investigate the regularity near 0 in each case; {\em e.g.,} in
Finsler-related examples it usually turns out to be $C^1$.
\end{remark}

Next, choose a smooth function $\ve :  M \to \R$ such that $\ve(p) \in
(-\ve_p,\ve_p)$ for every $p \in M$.  (The smoothness of $\ve$ is for our
later convenience: we want $\exp^\ve_p$ to be smooth in $\ve$ as well as
in all other parameters.)  Then the global map $\exp^\ve$ is defined
pointwise by $(\exp^\ve)_p = \exp^{\ve(p)}_p$.  The domain of $\exp^\ve$
is a tubular neighborhood of the 0-section in $TM$ and the graph of $\ve$
lies in a tubular neighborhood of the 0-section in the trivial line bundle
$\R \times M$.

We have an example, given to us by J. Hebda, to show that it is possible
that $\ve_p<1$ for every open neighborhood of $0\in T_p M$ if the SODE is
inhomogeneous.
\begin{example}
Consider the SODE on $\R$ given by
$$\ddx(t) = \pi\left(1 + \dx(t)^2 \right).$$
To integrate, we rewrite this as
$$\frac{d\dx}{1+\dx^2} = \pi\, dt$$
and obtain
$$\arctan\dx = \pi\, t+C_1\,.$$
Thus
$$\dx(t) = \tan \left(\pi\, t + C_1\right), \quad\dx(0) = \tan C_1$$
so
$$x(t) = \log\bigl|\sec\left(\pi\, t + C_1\right)\bigr| + C_2\,.$$
For $C_1\ge 0$, $x$ cannot be continued beyond
\begin{eqnarray*}
\pi t + C_1 &=& \frac{\pi}{2}\,,\\[4pt]
t &=& \frac{1}{2} - \frac{C_1}{\pi} < 1\,.
\end{eqnarray*}
Therefore the usual exponential map of this SODE is not defined ({\em
i.e.,} at $t=\ve =1$) for all $C_1\ge 0$.
\end{example}

The closer the graph of $\ve$ gets to the 0-section of $\R \times M$, the
larger the tubular neighborhood of the 0-section in $TM$ gets.
\begin{proposition}
For $\ve_1 < \ve_2$, we have $\dom(\exp^{\ve_{\!\!\;1}}) \supset
\dom(\exp^{\ve_{\!\!\;2}})$, attaining all of $TM$ for $\ve = 0$ when
$\exp^0 = \pi$.\eop
\label{exp0}
\end{proposition}
This puts the bundle projection $TM\surj M$ in the interesting position of
being a member of a one-parameter family of maps, all of whose other members
are local diffeomorphisms. (This is reminiscent of singular perturbations.)
\begin{theorem}
For every\/ $\ve$ such that\/ $0<|\ve|<\ve_p$, the generalized exponential
map\/ $\exp^\ve_p$ is a diffeomorphism of an open neighborhood of\/ $0 \in
T_p M$ with an open neighborhood of\/ $p \in M$.
\end{theorem}
\begin{proof}
This follows from the flow theorems in ODE ({\em e.g.,} \cite[pp.\,175,
302]{HS}) and a slight generalization of the usual argument ({\em e.g.,}
\cite[p.\,116f\,]{BJ}).  Note that for $v \in T_p M$, $\exp^\ve_p v =
\pi\P(\ve,v)$ where $\P$ is the local flow of $S$. Then on the 0-section
of $TM$, the induced tangent map $(\pi,\exp^\ve)_*$ in block form is given
by
$$ \left[ \begin{array}{cc} 0 & A\\ I & I \end{array} \right] $$
where $A$ is invertible. (When $S$ is homogeneous and $\ve = 1$, then
$A=I$ as in the usual proof.)
\end{proof}
If desired, one could use the construction in the proof of Theorem 4.4
in \cite{DRP4} to obtain a more explicit form for this $A$.

For reference, we record the following obvious result.
\begin{lemma}
$\ve$ is a geodesic parameter; {\em i.e.,} the curve obtained by fixing
$v$ and varying $\ve$ is a geodesic through $p$.\eop
\end{lemma}
Now consider another parameter $a$ as in
$$\exp^\ve_p(av)\,.$$
In general, $a$ will {\em not\/} be a geodesic parameter; {\em i.e.,} the
curve obtained by fixing $\ve$ and $v$ and varying $a$ is {\em not\/} a
geodesic through $p$.  See Figures \ref{j1} and \ref{j2} for a comparison.
Also note that these $a$-parameter curves are the exponentials of radial
lines in $T_p M$.
\begin{figure}
\begin{center}
\leavevmode
\def\epsfsize#1#2{.5#1}
\epsffile{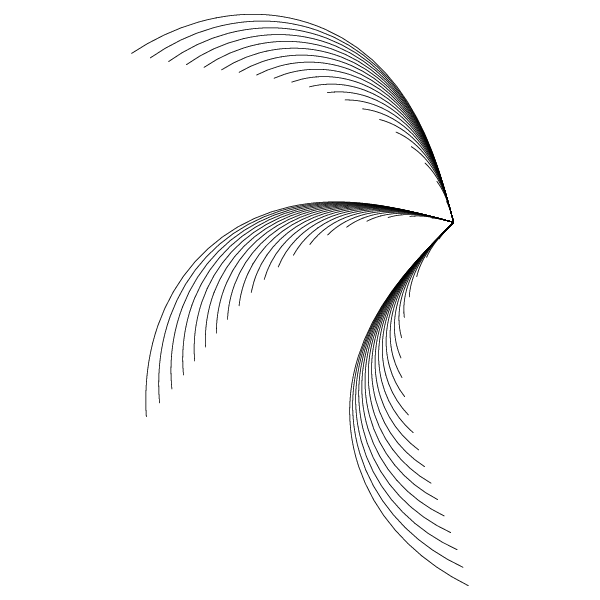}
\caption{\small curves $\exp^\ve_p(av)$ --- Each black curve is a geodesic
with $0 < \ve < 3$ and $a$ and $v$ fixed.  From shortest to longest in each
plume, $a$ steps in increments of 0.05 from 0.05 to 1. In each plume, $v$ is
constant.  There are three implicit $a$-parameter curves readily located,
one along the endpoints of each of the three plumes.}\label{j1}
\end{center}
\end{figure}\begin{figure}
\begin{center}
\leavevmode
\def\epsfsize#1#2{.6#1}
\epsffile{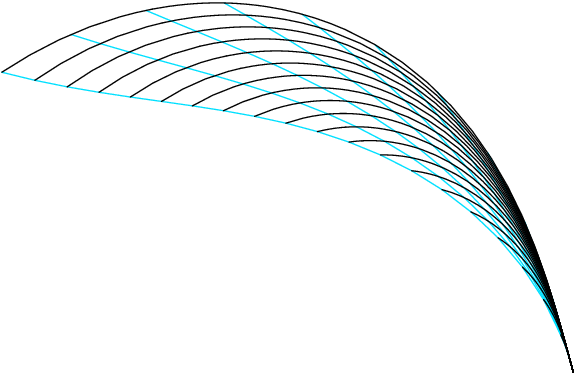}
\caption{\small curves $\exp^\ve_p(av)$ --- This is one plume from Figure
\protect\ref{j1}. Each black curve is a geodesic and each gray (blue)
curve is an $a$-parameter curve.  The new Jacobi fields are {\em along\/}
the black curves but {\em tangent\/} to the gray curves.}\label{j2}
\end{center}
\end{figure}\begin{proposition}
If $S$ is homogeneous, then $a$ as above is a geodesic parameter.
\end{proposition}
\begin{proof}
When $S$ is homogeneous, we can take $\ve=1$ and recover the usual
exponential map, and then $a$ is the usual geodesic parameter.
\end{proof}
The $a$-parameter curves are interesting:  they are the integral curves for
our new Jacobi vector fields. These were mentioned in \cite{DRP2} and will
be studied in more detail later \cite{DRP6}.  For now, we have the
following example.
\begin{example}
In $\R^2$, consider the SODE given by $\S^i(x,y) = y^i$ for $i=1,2$. The
geodesics are easily found to be $c(t) = v e^t + p$ where $v$ is the
initial velocity and $p$ is the initial position.  We can use the usual
exponential map since these curves are always defined for $t=1$. Thus we
obtain $\exp_{p}(v) = c(1) = v\,e + p$, regarding both $v$ and $p$ as
vectors in $\R^2$.

For the $a$-curves, we have $\exp_p(a\,v) = av\,e + p$, showing the
difference between the two types quite clearly: the geodesics have
exponential growth in velocity, while the $a$-curves have only linear
growth.
\end{example}
Finally, note that we could just as well define exponential-like maps
based on the $a$-curves and they would share most of the properties
of our new exponential maps.

\section{\heads Connections and their quasisprays}\label{cs}
A {\em (general) connection\/} on a manifold $M$ is a subbundle $\H$ of
the second tangent bundle $\pi_T:TTM\surj TM$ which is complementary to
the vertical bundle $\V$, so
\begin{equation}
TTM = \H\ds\V\, .
\label{ceq}
\end{equation}
The space of all connections on $M$ is denoted by $\econ(M)$, since this
definition is essentially due to Ehresmann \cite{E}.

Recall there are two vector bundle structures on $TTM$ over $TM$, denoted
here by $\pi_T$ and $\pi_*$.  While $\V$ is always a subbundle with
respect to both \cite[pp.\,18,20]{P}, $\H$ is a subbundle with respect to
$\pi_*$ if and only if the connection is linear \cite[p.\,32]{B}.

Also recall that quadratic sprays correspond to linear connections.  In
terms of the horizontal bundle $\H\! ,$ linearity is expressed as
$$\H_{av} = a_* \H_v $$
for $a\in\R$ considered as a map $TM\to TM$ and $v\in TM$. Thus one has
\begin{equation}
\H_{av} = a_*a^{m-1}\H_v
\label{hceq}
\end{equation}
as the second defining equation, together with (\ref{ceq}), of a
connection that is $h(m)$.
\begin{remark}
For an $h(m)$ semispray $S$ with integral $m$, Grifone's \cite{G1}
associated (generalized) connection coefficients $\Gamma$ are $h(m-1)$,
appropriately.  See (\ref{ours}) below for our version, which allows for
inhomogeneous connections.
\end{remark}

Here is the SODE induced by a connection.  We shall call it the {\em
geodesic quasispray\/} associated to the connection and its geodesics the
{\em geodesics\/} of the connection.
\begin{theorem}\label{cis}
For each connection $\H\! ,$ there is an induced SODE $S$ given by
$$S(v) = \pi_*\big|^{-1}_{\H_v} (v)\, ,$$
where $\pi:TM\surj M$ is the natural projection and $v\in TM$. We write
$\H\vdash S$ to denote this relationship.
\end{theorem}
\begin{proof}
As in the first paragraph of Poor's proof of 2.93 \cite[p.\,95]{P}, it is
easily verified that $S$ so defined is a SODE.  Indeed, $S$ is a section of
$\pi_*$ by construction, and $S$ is a section of $\pi_T$ because $\H$ is a
subbundle with respect to $\pi_T$.
\end{proof}
It is clear that this $S$ is horizontal, so compatible with the given
connection, and that it vanishes on the 0-section of $TM$.  This latter
means that constant curves, $c(t) = p\in M$ for all $t$, are degenerate
$S$-geodesics, a familiar property of geodesic sprays. Accordingly, we
shall refer to any SODE which vanishes on the 0-section of $TM$ as a
{\em quasispray}.

Unfortunately, when the connection is $h(m-1)$ this SODE is not
homogeneous {\em as a SODE;} it is only an $h(m)$ {\em vector field\/} on
$TM$.  In order to avoid this problem, we must consider a new type of
partial homogeneity for connections.
\begin{definition}\label{dvhc}
A connection $\H$ on $TM$ is {\em vertically homogeneous\/} of degree $m$,
denoted by $vh(m)$, if and only if
\begin{equation}\label{vhceq}
\H_{av} = a_* \av^{m-1} \H_v
\end{equation}
where $\av$ denotes scalar multiplication by $a$ in the vertical bundle
$\V$.
\end{definition}
Note that $h(m)$ and $vh(m)$ coincide only for $m=1$, the linear
connections.
\begin{proposition}\label{hgs}
If $\H$ is a connection with geodesic quasispray $S$, then $S$ is $h(m)$
if and only if $\H$ is $vh(m-1)$.
\end{proposition}
\begin{proof}
That $S$ is $h(m)$ if $\H$ is $vh(m-1)$ follows as in the second
paragraph of Poor's proof of 2.93 \cite[p.\,95]{P}, {\em mutatis
mutandis;} the converse results from a similar calculation.
\end{proof}

Connections may also be seen as sections of the bundle $G_H(TTM)$ of all
possible horizontal spaces, a subbundle of the Grassmannian bundle
$G_n(TTM)$.  To see what structure $G_H(TTM)$ has, consider $\R^{2n} =
\R^n\ds\R^n$ as the model fiber of $TTM$ and regard the first summand as
horizontal, the second as vertical.  With $GL_{2n}$ as the structure group
of $TTM$, we want the subgroup $A_H$ that preserves the vertical space and
maps any one horizontal space into another.  This can be conceived as
occurring in two steps.  First, we may apply any automorphisms of the
vertical and horizontal spaces separately.  Second, we may add vertical
components to horizontal vectors to obtain the new horizontal space.
$$ \arraycolsep=.5em \left[
\begin{array}{cc} I & 0 \\ \gl_n & I \end{array} \right] \cdot \left[
\begin{array}{cc} GL_n & 0 \\ 0 & GL_n \end{array} \right] $$
Our group $A_H$ is thus found to be a semidirect product entirely
analogous to an affine group.  The action is transitive and the right-hand
factor is the isotropy group of a fixed horizontal space, so the model
fiber for $G_H(TTM)$ is the resulting homogeneous space.  The induced
operation on representatives being given by
$$ \arraycolsep=.5em \left[
\begin{array}{cc} I & 0 \\ A & I \end{array} \right] \cdot \left[
\begin{array}{cc} I & 0 \\ B & I \end{array} \right]  = \left[
\begin{array}{cc} I & 0 \\ A+B & I \end{array} \right], $$
it follows that $G_H(TTM)$ is an affine bundle (bundle of affine spaces,
{\em vs.} vector spaces).  Thus a connection, being a section of this
bundle, provides a choice of distinguished point in each fiber, hence a
vector bundle structure on this affine bundle.

If we wish to consider only those connections compatible with a given
quasispray, we just replace arbitrary elements of $\gl_n$ with those
having a first column comprised entirely of zeros.  Note that this yields
an affine subbundle $G_H^S(TTM)$ of $G_H(TTM)$, with fibers being pencils
of possible horizontal spaces.
\begin{theorem}[extended APS]\label{cce}
Given a quasispray $S$ on $M$, there exists a compatible connection $\H$
in $TTM$.
\end{theorem}
Since the fibers of $G_H^S(TTM)$ are contractible, this is an easy
exercise in obstruction theory \cite[Ch.\,8]{DP}; however, an explicit
construction is desirable to provide a concrete representation for our
extension of the Ambrose-Palais-Singer correspondence, and we gave a
detailed proof in \cite{DRP4}.  For the convenience of the reader, we
repeat the complete proof.  First, we provide a brief sketch.  It mostly
follows the usual outline \cite[proof of Thm.\,2.98, pp.\,97ff\,]{P}, but
(as noted earlier) the exponential maps {\em do not\/} map radial lines in
the tangent space into geodesics in the base, so considerable extra care
is required to use correct pre-images of geodesics instead.
\begin{proof}
Let $\P$ denote the local flow of $S$ and $\g$ an integral curve of $S$
with $\g(0) = v \in T_p M$.  The basic idea is to use $S$ and $\P$ to {\em
define\/} notions of {\em horizontal\/} and {\em parallel\/} which will
coincide with the usual ones along $\g$ for any $\H \nd S$.  This is
essentially the same as the usual construction \cite{P}.  The problem is
that for inhomogeneous $S$, the ray $\{tv\}$ in $T_p M$ does {\em not\/}
exponentiate to a geodesic in $M$.

To remedy this, we proceed as follows. For each $v\in T_p M$, choose
$\ve_v$ so that $\exp_p^{\ve_{\!\!\:v}}v$ is defined. Such $\ve_v$ exist
by Proposition \ref{exp0}. For $0 \le t \le \ve_v$, define
\begin{equation}
\alpha_v(t) = \left(\exp_p^{\ve_{\!\!\:v}}\right)^{-1}\exp_p^t v\,.
\end{equation}
Then $\alpha_v(0) = 0$, $\alpha_v(\ve_v) = v \in T_p M$, and $\alpha_v$
exponentiates to the geodesic with initial condition $v$ at $p$. Note that
if $S$ is homogeneous, then $\alpha_v(t) = tv$.

We have a vector bundle map $\cJ:\pi^*TM\to \V$ which is an isomorphism on
fibers.  It is one version of canonical parallel translation on a vector
space, identifying the tangent space at each point with the vector space
itself.  Now, for each $w \in T_p M$ define
\begin{equation}
\H_w = \left\{ \left.\frac{d}{dt}\right|_{t=0} \pi_*
\P_{t*}\,\cJ_{\alpha_{\!\!\;v}(t)}w \Bigm| v \in T_p M \right\} .
\end{equation}
Clearly, this does not depend on the choices of $\ve_v$ made earlier.
(Note we are evaluating at 0.)  If $S$ is quadratic, it is easy to check
that this coincides with the usual construction as found in
\cite[pp.\,96--97]{P}, since in that case $\exp_p tv = \pi\P(t,v)$ for $v
\in T_p M$.  The proof that $\H$ so defined is a connection and that $\H
\nd S$ follows Poor's proof of 2.98 \cite[pp.\,97--99]{P} {\em mutatis
mutandis.}
\end{proof}
These connections will be our ``standard"---our generalization of
torsion-free linear connections; {\em viz.}\ equation (\ref{lcsg}),
Definition \ref{torsfree} and after.  In light of this, and the fact that
when applied to pseudoRiemannian geodesic sprays this construction yields
the Levi-Civita connection, we shall call them LC connections; {\em
cf.}\ Poor \cite[2.104 and 3.29]{P}.

\begin{remark}\label{c/q}
Note that the space of connections $\econ(M)$ fibers trivially over the
space of quasisprays $\qsp(M)$ since the latter has a vector space
structure, albeit not one compatible with that of all vector fields on
$TM$.
\end{remark}

\begin{remark}\label{c/ss}
Recall that any SODE on $TM-0$ is called a semispray. This is justified by
the fact that any construction such as ours that produces a compatible
connection over $TM$ from a quasispray there also produces one over
$TM-0$ from {\em every\/} SODE there.  In particular, this means that for
a SODE on $TM$ that is not a quasispray, the restriction of this SODE to
$TM-0$ is a semispray with a compatible connection over $TM-0$ even though
the original SODE did {\em not\/} have one over $TM$.  Such SODEs do not
seem to have been noted before, and further study of them is clearly
warranted.
\end{remark}

Here is an alternative, axiomatic characterization of a connection in terms
of the horizontal projection $H$.
\begin{itemize}
\item[\bf C1 ] $H$ is a smooth section of $\End(TTM)$ over $TM$.

\item[\bf C2 ] $H^2 = H$.

\item[\bf C3 ] $\ker H = \V$.
\end{itemize}
Then $\H = \im H$ is the horizontal bundle. Vertical homogeneity is
expressed with an optional axiom.
\begin{itemize}
\item[\bf Ch ] $H$ is $vh(m)$ if and only if $H_{av}a_* = a_*\av^{m-1}
H_v$ for all $v\in TM$ and $a\in\R$ ($v\in TM-0$ and $a\ne 0$ for $m<0$).
\end{itemize}
Homogeneous connections may be similarly axiomatized.

There is a natural vector bundle map $K:\V\to TM$ respecting $\pi_T$ which
is an isomorphism on fibers, a version of canonical parallel translation
of a vector space.  Using this, we define a connection map or connector
for an arbitrary connection and thence a covariant derivative.
\begin{definition}\label{kap}
For a connection $\H\! ,$ define the associated {\em connector\/} $\kappa
:TTM\to TM : z\mapsto K(z - H_v z)$ for $z\in T_v TM$.
\end{definition}
\begin{proposition}\label{kvbm}
The connector $\kappa$ is a vector bundle map respecting $\pi_T$ but\/ {\em
not} $\pi_*$ in general.  It respects $\pi_*$ if and only if the connection
is linear.
\end{proposition}
\begin{proof}
As in Poor \cite[p.\,72f\,]{P}, {\em mutatis mutandis.}
\end{proof}
According to Besse \cite[p.\,32f\,]{B}, a {\em symmetric\/} connector
(connection) is invariant under the natural involution $J$ of $TTM$.
Clearly this is possible only for linear connections.

Now we are ready for the main event. Let $V$ and $U$ be a vector fields on
$M$ with $V_p = v$ and $U_p = u$.
\begin{definition}\label{cd}
The {\em covariant derivative\/} associated to the connection $\H$ is
the operator defined by
$$\del{U}V = \kappa(V_* U) = K(V_* U - H_V V_* U)$$
and is tensorial in $U$ but {\em nonlinear\/} (in general) in $V$.
\end{definition}
This last comes from the general lack of respect for the $\pi_*$ structure
by $\H\! ,$ $H\!$, and $\kappa$.
\begin{example}\label{hex}
We always have $\del{0}V = 0$.  For all $vh(m)$ connections, $\del{U}aV =
K(a_*V_*U - H_{aV}a_*V_*U) = aK(V_*U - \av^{m-1}H_V V_*U)$, and similarly
for homogeneous ones. So (vertically) homogeneous connections do not
differ significantly from linear ones.  In particular, $\del{U}0 = 0$ for
all $U$ for all (vertically) homogeneous connections; in fact, they all
have the same horizontal spaces along the 0-section of $TM$, namely the
subspaces tangent to it ({\em i.e.,} those in the image of $0_* : TM\to
TTM$).  We call all such connections sharing this property {\em
0-preserving;} they differ minimally from (vertically) homogeneous
(including linear) connections.  In contrast, connections with $\del{U}0
\ne 0$ for even some $U$ are much farther from linear; we call them {\em
strongly nonlinear.} See Figure \ref{sch} for a schematic view.
\end{example}
\begin{figure}
\begin{center}
\setlength{\unitlength}{.15em}
\begin{picture}(200,75)
\put(10,50){\line(1,0){180}}
\put(100,50){\makebox(0,0){$\bullet$}}
\put(55,45){\makebox(0,0){\footnotesize homogeneous}}
\put(145,42){\makebox(0,0){$\mbox{\footnotesize vertically}\atop
   \mbox{\footnotesize homogeneous}$}}
\put(97.5,54){\psrotate{\mbox{\footnotesize linear}}}
\put(100,53){\oval(200,72)}
\put(100,12){\makebox(0,0){\footnotesize 0-preserving}}
\end{picture}\end{center}
\caption{\small Each set of connections is closed with empty interior in
the next:  linear in homogeneous, linear in vertically homogeneous, linear
and homogeneous in 0-preserving, linear and vertically homogeneous in
0-preserving, linear and homogeneous and vertically homogeneous in
0-preserving, 0-preserving in the whole.  The strongly nonlinear
connections may be visualized as a 3-d cloud containing the 0-preserving
ones.}\label{sch}
\end{figure}
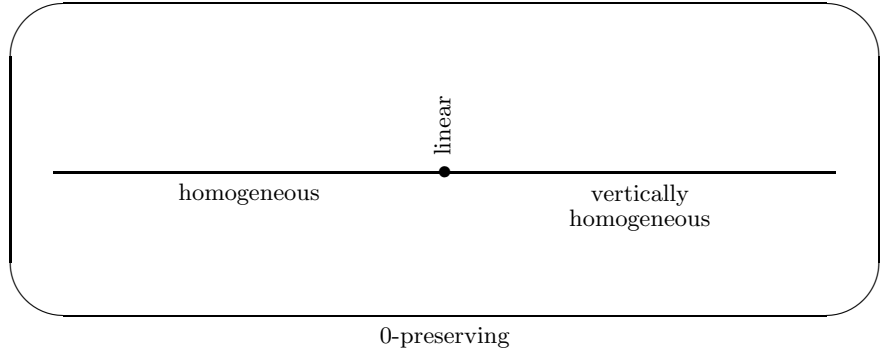

As usual, $\mathfrak{X}$ denotes the vector fields on $M$.  There is also
a natural vector bundle map $\cJ:\pi^*TM\to \V$ which is an isomorphism on
fibers, another version of canonical parallel translation on a vector
space.
\begin{theorem}\label{corresp}
There is a bijective correspondence between (possibly nonlinear)
connections $\H$ and our (possibly nonlinear) covariant derivatives $\D$
on $TM$.
\end{theorem}
\begin{proof}
It suffices to show that we can reconstruct $\H$ from its associated
covariant derivative $\D$.  For each $u\in T_p M$, define
$$ \pbar{\H}_u = \{U_*v - \cJ_u\del{v}U\mid U\in\mathfrak{X}, U_p = u, v\in
T_pM\} $$
and form the subbundle $\pbar\H$ in $TTM$ in the obvious way.  It is easy
to see that $\pbar\H$ is complementary to $\V$ as required, hence a
connection.  That $\pbar\H$ is smooth is straightforward.  Finally, $\pbar\H
= \H$ from this construction and the construction of $\D$ from $\H\!$
\cite{DRP4}.
\end{proof}
Compare \cite[p.\,77, proof of 2.58]{P}.  Thus as usual, we may refer
indifferently to $\H$ or its associated $\D$ as the connection.

Generalized connection coefficients may be introduced through
\begin{equation}
\left(KH_V V_*U\right)^k = \Gamma^k_i(V)U^i\,,\label{ours}
\end{equation}
making manifest the tensoriality in $U$. Here is an example of their
use. Observe that $(KV_* U)^k = U^i\partial_i V^k$ so that
\begin{equation}
\left(\del{U}V \right )^k =  U^i\partial_i V^k - \Gamma^k_i(V)U^i
\label{lccd}
\end{equation}
is the covariant derivative.

We find the usual relation between the two notions of geodesic.
\begin{theorem}\label{cg=sg}
A curve $c$ is a geodesic of\/ $\H$ if and only if\/ $\del{\dc}\,\dc = 0$.
\end{theorem}
\begin{proof}
$\del{\dc}\,\dc = \kappa(\dc_*\dc) = K(\dc_*\dc - H_{\dc}\,\dc_*\dc) =
K(\dc_*\dc - S(\dc))$ by the construction of $S$ in Theorem \ref{cis}.
Now all we have to do is identify $\dc_*\dc$ as $\ddc$ and recall that $K$
is an isomorphism on fibers.
\end{proof}
If we are given the geodesic equation of $\H$ in the form
\begin{equation}
\ddot{c}^{\,k} = \Gamma^k_i(\dot{c})\dot{c}^i ,
\label{lcgeq}
\end{equation}
then
\begin{equation}
\S^k(\dot{c}) = \Gamma^k_i(\dot{c})\dot{c}^i
\label{lcsg}
\end{equation}
gives the quasispray $S$ induced by the connection $\H$.  Using these
connection coefficients, we obtain the LC connection associated to $S$
by our extended APS construction; see also Theorem \ref{=geos}.

Curvature is readily handled.  Let $\H$ be a connection on $M$.  The {\em
horizontal lift\/} of a vector field $U$ on $M$ is defined as usual and
denoted by $\bU$.
\begin{definition}\label{curv}
Given vector fields $U$ and $V$ on $M$, the {\em curvature operator\/}
$R(U,V):TM\to TM$ is defined by
$$R(U,V)w = \kappa\left([\bV,\bU]_w\right)$$
for all $w\in TM$.  It is tensorial in the first two arguments, but
{\em nonlinear\/} (in general) in the third.
\end{definition}
The arguments are reversed on the right in order to obtain the usual formula
in terms of the associated covariant derivative,
$$R(U,V)W = \del{U}\del{V}W - \del{V}\del{U}W - \del{[U,V]}W\, ,$$
as one may verify readily.  It is also easy to check that this curvature
vanishes if and only if $\H$ is integrable, thus justifying our definition.

Torsion is considerably more obscure.  Consider two (possibly nonlinear)
connections $\pbar\H$ and $\H$ on $TM$ with corresponding (possibly
nonlinear) covariant derivatives $\pbar\D$ and $\D$.
\begin{definition}\label{dif}
Given two covariant derivatives $\pbar\D$ and $\D$, define the {\em
difference operator\/} $\cD = \pbar\D - \D$.
\end{definition}
We think of $\cD$ as having two arguments, $\cD(U,V) = \bdel{U}V -
\del{U}V$.  It is always tensorial in $U$, but is {\em nonlinear\/} (in
general) in $V$.

We define the {\em covariant differential\/} as usual {\em via\/} $(\D V)U
= \del{U}V$. As an operator, $\D V$ is still linear in its argument $U$.
\begin{lemma}\label{bar}
For all $v\in TM$, $\pbar{\H}_v = \{z - \cJ_v\cD(\pi_* z,v)\mid
z\in\H_v\}$.
\end{lemma}
\begin{proof}
Let $v\in T_p M$, $z\in\H_v$, $V\in\mathfrak{X}$ such that $(\D V)_p = 0$ and
$V_p = v$. Thus if $u = \pi_*z\in T_p M$, then $z=V_*u\in\H_v$. Now
$$ \bar\kappa V_*u = \bdel{u}V = \del{u}V + \cD(u,V) = \cD(u,V) =
\bar\kappa\cJ_v\cD(u,V)\,,$$
so $\bar\kappa\left( z - \cJ_v\cD(u,v)\right) = 0$ and $z - \cJ_v\cD(u,v)
\in \pbar{\H}_v$.

Since $\pi_*$ is an isomorphism of the horizontal spaces $\pbar\H_v$ and
$\H_v$ with $T_p M$ and $\pi_*z = \pi_*\left( z - \cJ_v\cD(u,v)\right)$,
this yields all of $\pbar\H_v$.
\end{proof}
Compare this next result with \cite[Prop.\ on p.\,99]{P}.
\begin{theorem}\label{=geos}
Two connections on $TM$ have the same geodesic quasispray if and only if
their associated difference operator is alternating (vanishes on the
diagonal of\/ $TM\ds TM$).
\end{theorem}
\begin{proof}
For each $v\in TM$, $S_v = \pi_*\big|^{-1}_{\H_v} (v)$ while $\pbar{S}_v
= \pi_*\big|^{-1}_{\pbar{\H}_v} (v) = \pi_*\big|^{-1}_{\H_v} (v) -
\cJ_v\cD(v,v)$. Therefore $\pbar{S} = S$ if and only if $\cD(v,v) = 0$ for
all $v\in TM$.
\end{proof}
For {\em linear\/} connections, $\cD$ is {\em bilinear\/} and alternating
is equivalent to antisymmetric (or, skewsymmetric).  In general, of
course, this does not hold.

The familiar formula for torsion $T(U,V) = \del{U}V - \del{V}U - [U,V]$ is
not linear (let alone tensorial) in either argument.  Thus the usual trick
to get a torsion-free linear connection, replacing $\D$ by $\pbar\D = \D -
\half T$, will not work for our nonlinear connections.  Indeed, $\pbar\D$
and $\D$ seem to have the same geodesics and $\pbar\D$ is formally
torsion-free, {\em but\/} the new $\pbar\D$ is {\em not\/} one of our
nonlinear covariant derivatives:  $\bdel{U}V$ is {\em not\/} tensorial in
$U$.

A replacement $\cT$ for torsion must also be alternating in order for it
to play the same role in general that torsion does for linear connections.
For then, given such a $\cT$, $\pbar\D = \D + \cT$ is another nonlinear
covariant derivative of our type with the same geodesics as $\D$; or,
with the same geodesic qspray as $\D$.

What we shall do is one of the classic mathematical gambits:  turn a
theorem into a definition.
\begin{definition}
We define the LC connections constructed in the proof of Theorem
\ref{cce} to be the {\em torsion-free\/} connections.
\label{torsfree}
\end{definition}
Equivalently, we are regarding the usual torsion formula as derived from
the difference operator (difference tensor in the linear case)
construction \cite[pp.\,99--100]{P}.  See also Poor
\cite[pp.\,101--102]{P} for the relation to the classic
Ambrose-Palais-Singer correspondence and compare to \cite[2.104]{P}.

Now we may construct the torsion of a (possibly nonlinear) connection $\H$
with corresponding (possibly nonlinear) covariant derivative $\D$.  By
Theorem \ref{cis}, $\H$ induces a (unique) quasispray $S$.  Use the proof
of Theorem \ref{cce} to construct the connection $\hat\H$ from $S$.  By
Theorem \ref{corresp} there is a unique covariant derivative $\hat\D$
corresponding to $\hat\H$.  Let $\cD = \D - \hat\D$ be the difference
operator, so $\hat\D = \D - \cD$ is torsion-free.
\begin{definition}
Using the preceding notations, the (generalized) {\em torsion\/} of $\D$
is defined by $\cT = 2\cD = 2\left(\D - \hat\D\right)$.
\label{tors}
\end{definition}
The factor of two here and the subtraction order make verification that
this reduces to classical torsion in the linear case immediate, and
preserves the traditional formula $\hat\D = \D - \half\cT$ for the
associated torsion-free connection. See Poor \cite[2.105]{P} for how this
fits into the classical APS correspondence.

\section{\heads Finsler spaces}\label{fs}
For the benefit of those readers not familiar with Finsler geometry, we
offer a few introductory and historical remarks.

Finsler spaces are manifolds whose tangent spaces carry a norm (rather
than an inner product; {\em cf.}\ Banach {\em vs.}\ Hilbert spaces) that
varies smoothly with the base point.  Although Riemann actually defined
such spaces in his 1854 {\em Habilitationsvortrag,} the modern name comes
from P.\ Finsler's thesis of 1918 in which he studied the variational
problem in regular metric spaces.

Geometric objects on a Finsler space depend not only on the base point but
also on the fiber component.  Classically, a Finsler metric is given by a
fundamental function $F$ which is continuous on $TM$, smooth and positive
on $TM-0$, and positively homogeneous of degree one in the fiber
component.  An orthogonal structure on the vertical bundle is defined by
the vertical Hessian of the square of the fundamental function.  A
differentiable manifold $M$ with a Finsler metric is called a Finsler
space.  One modern variation is to consider only a subset of $TM$ as the
domain of $F$, with appropriate changes to the rest of the definition.

We define the Finsler functions $L$, the {\em basic\/} function, and the
traditional $F$, the {\em fundamental\/} function, following two of the
seemingly overlooked but prescient papers of Beem \cite{B1,B2}.

We require $L$ to be $h(2)$ and note that it corresponds to $F^2$, but to
get pseudoRiemannian structures we must require only that $L$ be real
valued, {\em not\/} strictly positive, else we could not have spacelike,
timelike, and null geodesics, as first observed by Beem \cite{B1}.  We
also require that $L$ be continuous on $TM$ and smooth on $TM-0$,
following tradition.

Then we use $|L|^{\frac{1}{2}}$ as the correspondent to $F$; {\em e.g.,}
in the first variation formula ({\em viz.}\ \cite[Chapt.\,10]{O}) to
obtain non-null geodesics. We shall see later how to obtain the null
geodesics.

The {\em vertical Hessian}
\begin{equation}
g_{ij}(y) = \frac{1}{2}\frac{\partial^2}{\partial y^i \partial y^j}L(y)
\label{vH}
\end{equation}
is traditionally assumed positive definite, which perforce yields only
Riemannian entities, such as the traditional orthogonal structure on the
vertical bundle $\V(TM-0)$.  We shall merely assume it is nondegenerate,
allowing pseudoRiemannian entities.  Together with our relaxed condition
on $L$, this gives us pseudoFinsler (or indefinite Finsler) structures as
first defined by Beem around 1969 \cite{B1}.

The traditional geodesic coefficient is \cite{BCS}
$$G^i(y) = \frac{1}{2} g^{il}(y) \left[
\frac{\partial}{\partial x^k \partial y^l}L(y)y^k - \frac{\partial}
{\partial x^l}L(y) \right].$$
To be consistent with our conventions, we take the negative of this for
our geodesic coefficients,
\begin{equation}
\eG^i(x,y) = -G^i(x,y)
\label{gc}
\end{equation}
where we have restored the explicit $x$ and $y$ dependence.  These
components $\eG^i$ then make up a semispray function $\eG$ with
accompanying $h(1)$ geodesic semispray $G$.  In induced local
coordinates,
$$ G:(x,y) \mapsto (x,y,y,\eG(x,y))\,.$$
The traditional Finsler geodesic equations are
$$\ddc^{i} + G^{i}(\dc) = 0\,.$$
In our notation and conventions, this becomes
\begin{equation}
\ddc = G(\dc)\,.
\label{fge}
\end{equation}

The traditional nonlinear connection coefficients are
$$N^i_j = \frac{\partial}{\partial y^j}G^i.$$
Converting to our notation and formalism, we obtain the $vh(0)$ nonlinear
connection on $TM-0$ given locally by
\begin{equation}
\G^i_j(x,y) = \frac{\partial}{\partial y^j}\eG^i(x,y)\,.
\label{fcnc}
\end{equation}
In fact, this last equation holds in complete generality, as can be seen
easily from (\ref{lcsg}). We chose to take note of it here in recognition
of the historical context.

Once we have the (nonlinear) connection $\H$ determined by $\G$, we obtain
the associated (nonlinear) covariant derivative $\D$ from Definition
\ref{cd}; it is unique by Theorem \ref{corresp}.  Using this connection,
we may then recoup (Theorem \ref{cg=sg}) all the (timelike and spacelike)
geodesics found in Finsler geometry tradition {\em via\/} the First
Variation, and we also obtain all the null geodesics, which {\em cannot}
\cite[Chapt.\,10]{O} be so found.  Therefore, as first noted by Beem
\cite{B2}, we do indeed have genuine pseudoFinsler geometry.

\section{\heads Geodesic connectivity and stability}\label{gcs}

In \cite{DRP1}, we defined a SODE to be LD if and only if its usual
exponential map is a local diffeomorphism.  For some results there, we used
the fact that the geodesics of such SODEs give normal starlike
neighborhoods of each point in $M$.  (In fact, the $a$-curves also give
such neighborhoods, as is easily seen.) Thanks to our new exponential
maps (Section \ref{exp}), these results now immediately extend
to all SODEs.  For convenience, we state them here.
\begin{proposition}
Let $M$ be a manifold with a pseudoconvex and disprisoning SODE $S$. If\/
$S$ has no conjugate points, then $M$ is geodesically connected.
\label{p5}
\end{proposition}

Let $M$ be a manifold with a SODE $S$ and let $\widetilde {M}$ be a
covering manifold.  If $\phi:\widetilde{M}\rightarrow M$ is the covering
map, then it is a local diffeomorphism.  Thus $\tilde{S} =
(\phi_{\ast})^{\ast}S$ is the unique SODE on $\widetilde{M}$ which covers
$S$, geodesics of $\tilde{S}$ project to geodesics of $S$, and geodesics of
$S$ lift to geodesics of $\tilde{S}$.  Also, $S$ has no conjugate points if
and only if $\tilde{S}$ has none.  The fundamental group is simpler, and
$\tilde{S}$ may be both pseudoconvex and disprisoning even if $S$ is
neither.
\begin{corollary}
Let $M$ be a manifold with a pseudoconvex and disprisoning SODE $S$ and let
$\widetilde{M}$ be a covering manifold with covering SODE $\tilde{S}$.  If
$\tilde S$ has no conjugate points, then both $\widetilde M$ and $M$ are
geodesically connected.  \label{qcov}
\end{corollary}
\begin{theorem}
Let $S$ be a pseudoconvex and disprisoning SODE on $M$.  If $S$ has no
conjugate points, then for each $p \in M$ the exponential maps of $S$ at $p$
are diffeomorphisms.
\end{theorem}
We remark that none of these results require (geodesic) completeness of
the SODE $S$.

We now consider the joint stability of pseudoconvexity and disprisonment
for SODEs in the fine topology.  Because each linear connection determines
a (quadratic) spray, Examples 2.1 and 2.2 of \cite{BP4} show that neither
condition is separately stable.  (Although \cite{BP4} is written in terms
of principal symbols of pseudodifferential operators, the cited examples
are actually metric tensors).  We shall obtain $C^{0}$-fine stability,
rather than $C^{1}$-fine stability as in \cite {BP4}, due to our effective
shift from potentials to fields as the basic objects.  The proof requires
some modifications of that in \cite {BP4}; we shall concentrate on the
changes here and refer to \cite {BP4} for an outline and additional
details.

Rather than considering $r$-jets of functions, we now take $r$-jets of
sections in defining the Whitney or $C^r$-fine topology as in Section 2 of
\cite{BP4}.  Let $h$ be an auxiliary complete Riemannian metric on $M$.
Thus we look at the $C^r$-fine topology on the sections of $TTM$ over $TM$.

If $\g_1$ and $\g_2$ are two integral curves of a SODE $S$ with $\g_1(0) =
(x,v)$ and $\g_2(0) = (x,\l v)$ for some positive constant $\l$, then the
inextendible geodesics $\p\circ\g_1$ and $\p\circ\g_2$ no longer differ only
by a reparametrization.  Thus, in contrast to \cite{BP4}, we must now
consider an integral curve for each non-zero tangent vector at each point of
$M$.  Note this also means that we can no longer use the $h$-unit sphere
bundle to obtain compact sets in $TM$ covering compact sets in $M$.

Observe that the equations of geodesics involve no derivatives of $S$.
Thus if $\g :[0,a]\rightarrow TM$ is a fixed integral curve of $S$ in $TM$
with $\g(0) = v_0 \in TM$ and if $\g' :[0,a]\rightarrow TM$ is an integral
curve of $S'$ in $TM$ with $\g'(0) = v$, then $d_h \left( \p\circ\g(t),
\p\circ\g'(t)\right) < 1$ for $0\le t\le a$ provided that $v$ is
sufficiently close to $v_0$ and $S'$ is sufficiently close to $S$ in the
$C^0$-fine topology.  This and the $\sigma$-compactness of $T K_1$ when
$K_1$ is compact yield the following result.
\begin{lemma}
Assume $K_{1}$ is a compact set contained in the interior of the compact
set $K_{2}$, $V$ is an open neighborhood of $K_2$, $S$ is a disprisoning
SODE, and let $\epsilon > 0$.  There exist countable sets $\{v_i\}
\subseteq TK_{1}$ of tangent vectors and $\{\delta_{i}\}$ and $\{a_{i}\}$
of positive constants such that if $S'$ is in a $C^{0}$-fine
$\epsilon$-neighborhood of $S$ over $V$, then the following hold:
\begin{enumerate}
\item if $c$ is an inextendible $S$-geodesic with $c(0)$ in a
$\delta_{i}$-neighborhood of $v_{i}$, then $c[0,a_{i}] \subset V$ and $c
(a_{i}) \in V-K_{2}$;

\item If $c^{\prime}$ is an inextendible $S^{\prime}$-geodesic with
$\dot{c}^{\prime}(0)$ in a $\delta _{i}$-neighborhood if $v_{i}$, then
$c^{\prime}[0,a_{i}] \subset V $ and $c^{\prime}(a_{i}) \in V-K_{2}$;

\item Two inextendible geodesics, $c$ of $S$ and $c^{\prime}$ of
$S^{\prime}$ with $\dot{c}(0)$ and $\dot{c}'(0)$ in a $\delta
_{i}$-neighborhood of $v_{i}$, remain uniformly close together for $0\leq t
\leq a_{i}$;

\item The union of all the $\delta_{i}$-neighborhoods of the $v_{i}$
covers $TK_{1}$.\eop
\end{enumerate}
\end{lemma}

Continuing to follow \cite{BP4}, we construct the increasing sequence of
compact sets $\{ A_n\}$ which exhausts $M$ and the monotonically
nonincreasing sequence of positive constants $\{ \epsilon_n\}$.  The only
additional changes from \cite[p.\,17f\,]{BP4} are to use integral curves of
$S$ in $TM$ instead of bicharacteristic strips in $T^*M$.  No other
additional changes are required for the proof of the next result either.
\begin{lemma}
Let $S$ be a pseudoconvex and disprisoning SODE and let $S^{\prime}$ be
$\delta$-near to $S$ on $M$.  If $c^{\prime}:(a,b)\rightarrow M$ is an
inextendible $S'$-geodesic, then there do not exist values $a < t_{1} <
t_{2} < t_{3} < b$ with $c^{\prime} (t_{1}) \in A_{n}$, $c ^{\prime} (t_{3})
\in A_{n}$, and $c^{\prime} (t_{2}) \in A_{n+4} - A_{n+3}$.\eop \label{3.2}
\end{lemma}

Now we establish the stability of pseudoconvex and disprisoning
SODEs by showing that the set of all SODEs in $\sode(M)$ which
are pseudoconvex and disprisoning is an open set in the $C^0$-fine
topology.  The only changes needed from the proof of Theorem 3.3 in
\cite[p.\,19]{BP4} are replacing principal symbols by SODEs,
bicharacteristic strips by integral curves, $S^*A_n$ by $TA_n$, and
references to Lemma 3.2 there by references to Lemma~\ref{3.2} here.
\begin{theorem}
If $S\in\sode(M)$ is pseudoconvex and disprisoning, then there is some
$C^{0}$-fine neighborhood $W(S)$ in $\sode(M)$ such that each $S' \in
W(S)$ is both pseudoconvex and disprisoning.\eop \label{spd}
\end{theorem}
\begin{corollary}
If $M$ is a pseudoconvex and disprisoning pseudoRiemannian manifold, then
all (possibly nonlinear) connections on $M$ which are sufficiently close
to the Levi-Civita connection are also pseudoconvex and disprisoning.\eop
\end{corollary}

\end{document}